\documentclass[12pt,a4paper]{article}
\usepackage{amssymb,amsmath,amsthm}
\usepackage{latexsym}
\newcommand{\essup}[1]{{\rm ess}\,{{\displaystyle \sup_{\hspace*{-6mm}{#1}}}}\,}
\newcommand{\Int}[2]{{\displaystyle \int_{ #1}^{ #2}}}
\usepackage{color}

\newtheorem{theorem}{Theorem}
\newtheorem{lemma}{Lemma}
\newtheorem{proposition}{Proposition}

\newtheorem{remark}{Remark}
\numberwithin{equation}{section}
\numberwithin{theorem}{section}
\numberwithin{lemma}{section}
\numberwithin{proposition}{section}
\numberwithin{corollary}{section}
\numberwithin{remark}{section}
\usepackage[mathscr]{eucal}
\begin{document}
\title{Attainability of Time-Periodic flow of a Viscous Liquid Past an Oscillating Body}
\author{Giovanni P. Galdi\thanks{Partially supported by NSF grant DMS-1614011}\ \ \,and Toshiaki Hishida\thanks{
Partially supported by the Grant-in-Aid for Scientific Research 18K03363
from JSPS}}
\maketitle
\begin{abstract}
A body $\mathscr B$ is started from rest by  translational motion in an otherwise quiescent Navier-Stokes liquid filling the whole space. We show, for small data, that if after some time $\mathscr B$ reaches a spinless oscillatory motion of period $\cal T$, the liquid will eventually execute also a time periodic motion with same period $\cal T$. This problem is a suitable generalization of the famous Finn's starting problem for steady-states, to the case of time-periodic motions.
\end{abstract}

\section{Introduction
}
\label{intro}
Consider a rigid body, $\mathscr B$,  at rest and completely immersed in a quiescent Navier-Stokes liquid filling the whole three-dimensional space, $\Omega$, outside $\mathscr B$. Next, suppose  that, at time $t=0$ (say), $\mathscr B$ is smoothly set in translational motion (no spin) and that after the time  $t=1$ (say), its velocity $\eta=\eta(t)$ coincides with a periodic function, $\xi=\xi(t)$, of  period $\mathcal T$ whose average over  the time interval $[0,\mathcal T]$ vanishes. In the particular case where both $\eta$ and $\xi$ are parallel to a given direction, the above means that $\mathscr B$ is brought from rest to a  regime where it oscillates between two fixed configurations.  In the general case, $\mathscr B$ is taken from rest to a (spinless) motion where its center of mass  moves periodically
along a given closed curve.\par 
On physical grounds, it is expected that, under the given assumptions,  the liquid will eventually reach a time-periodic flow of period $\mathcal T$, at least if the magnitude of $\eta$ and (possibly) some of its derivatives is not ``too large." This specific circumstance is often referred to as {\em attainability property} of the flow. 
In this regard, it is worth mentioning a famous problem of attainability, the so called ``Finn's starting problem" \cite{Fi} where $\mathscr B$    accelerates (without spinning) from rest to a given constant translational velocity. In such a case, the terminal flow of the liquid is expected to be steady-state. Finn's problem  was eventually and
affirmatively solved by Galdi, Heywood and Shibata \cite{GaHS} and,  with more general assumptions, very recently by Hishida and Maremonti \cite{HiMa}.  
\par
In analogy with these results, the main objective of this paper will be to show that, under  the given hypothesis on the motion of $\mathscr B$,   the liquid indeed attains a corresponding time-periodic flow of period $\mathcal T$. 
\par
We shall next give a rigorous mathematical formulation of the problem.   
Let us begin to observe that
the translational velocity $\eta(t)$ can be written as
\[
\eta(t)=h(t)\xi(t),
\]
where we assume 
\begin{equation}
\begin{split}
&\xi(t+{\cal T})=\xi(t) \quad (t\in\mathbb R), \qquad
\int_0^{\cal T}\xi(t)\,dt=0, \\
&\xi\in W^{3,2}(0,{\cal T})=W^{3,2}(0,{\cal T}; \mathbb R^3),
\end{split}
\label{osc}
\end{equation}
and 
\begin{equation}
h\in C^1(\mathbb R; \mathbb [0,1]), \qquad
h(t)=0 \quad (t\leq 0), \qquad
h(t)=1 \quad (t\geq 1).
\label{transi}
\end{equation}
The governing equations of the liquid flow, driven by the translational velocity $\eta$   of the body, are thus given by
\begin{equation}
\begin{array}{ll}\smallskip\left.\begin{array}{ll}\smallskip
\partial_tu+u\cdot\nabla u=\Delta u+\eta(t)\cdot\nabla u-\nabla p_u, \\
\mbox{div $u$}=0\end{array}\right\}\ \ \mbox{in $\Omega\times (0,\infty)$}, \smallskip\\
u|_{\partial\Omega}=\eta(t),\smallskip  \\
u\to 0 \quad\mbox{as $|x|\to\infty$}, \\
u(\cdot,0)=0,
\end{array}
\label{NS}
\end{equation} 
where $u=u(x,t)$ and $p_u=p_u(x,t)$ are, respectively,
the velocity vector field and pressure field of the liquid,
and $\Omega$ (the exterior
of the body in $\mathbb R^3$) is assumed to have a sufficiently smooth boundary $\partial\Omega$.
Likewise, if the translational velocity of $\mathscr B$ is the time-periodic function $\xi$, 
it is reasonable to expect that, the corresponding velocity field of the liquid $v=v(x,t)$ is
time-periodic of period $\mathcal T$ ($\mathcal T$-periodic) as well, and obeys the following equations
\begin{equation}
\begin{array}{ll}\smallskip\left.\begin{array}{ll}\smallskip
\partial_tv+v\cdot\nabla v=\Delta v+\xi(t)\cdot\nabla v-\nabla p_v, \\
\mbox{div $v$}=0\end{array}\right\}\ \ \mbox{in $\Omega\times \mathbb R/{\cal T}\mathbb Z$},\smallskip \\
v|_{\partial\Omega}=\xi(t), \smallskip\\
v\to 0 \quad\mbox{as $|x|\to\infty$},
\end{array}
\label{periodic-NS}
\end{equation}
where $p_v$ denotes the pressure associated with $v$.

In \cite{Ga} the first author
showed  existence, uniqueness and regularity of a
$\mathcal T$-periodic solution $(v(t),p_v(t))$ to \eqref{periodic-NS}
for all ``small" $\xi(t)$ satisfying \eqref{osc}.
Furthermore, he provided a detailed analysis of the asymptotic representation
of $v(t)$ at spatial infinity, by showing that
the leading term of $v(t)$ is given by a distinctive {\em steady-state} velocity field $U(x)$ that decays at large spatial distances like $|x|^{-1}$. Therefore, in general, $v(t)\not\in L^2(\Omega)$, for all $t\in\mathbb R$.

Let
\begin{equation}
u=hv+w,
\label{sum}
\end{equation}
where $h$ is the function given  in \eqref{transi}.
Then, from  \eqref{NS}, we deduce that the ``perturbation" $w(t)$ should obey,
together with the corresponding pressure $p_w=p_u-hp_v$, the following system of equations
\begin{equation}
\begin{array}{ll}\smallskip
\left.\begin{array}{ll}\smallskip
\partial_tw+w\cdot\nabla w+h(t)(v\cdot\nabla w+w\cdot\nabla v)\smallskip\\
\qquad\qquad\quad=\Delta w+\eta(t)\cdot\nabla w-\nabla p_w+f, \\
\mbox{div $w$}=0\end{array}\right\} \ \ \mbox{in $\Omega\times(0,\infty)$}, \\
w|_{\partial\Omega}=0, \smallskip\\
w\to 0 \quad\mbox{as $|x|\to\infty$}, \\
w(\cdot,0)=0,
\end{array}
\label{perturbed}
\end{equation}
with the forcing term ($h^\prime:=dh/dt$)
\begin{equation}
f:=-h^\prime v+(h-h^2)(v-\xi)\cdot\nabla v\,.
\label{force}
\end{equation}
The desired attainability property consists then in showing that the solution $w(t)$ to \eqref{perturbed} (exists, is unique and) tends to 0 as $t\to\infty$ in a suitable norm.  
In this respect, some comments are in order. Since $u(0)=0$, one would expect that the solution $u(t)$ to \eqref{NS} has 
finite energy, namely $u(t)\in L^2(\Omega)$ for all $t\ge 0$. Moreover, as noticed earlier on, $v(t)$ is, in general, {\em not} in $L^2$.
Consequently, in view of \eqref{sum}, $w(t)$ need not be in $L^2(\Omega)$, as also suggested by the fact that $f$ is {\em not} in $L^2(\Omega)$. This
 implies that   ``energy-based methods" might not be an appropriate tool to analyze the asymptotic behavior of $w(t)$, and one has thus to resort to the more general
$L^q$-theory.
This difficulty is analogous to that encountered in Finn's starting problem, which was in fact solved in \cite{GaHS} thanks to the asymptotic properties  of the Oseen
semigroup in $L^q$-spaces, proved for the first time in \cite{KS}.

However, in comparison with  \cite{GaHS}, our problem presents the following {\em two further} complications. (i)
The  velocity field $v(t)$, $t\in\mathbb R$, possesses weaker summability properties at large spatial distances than its steady-state counterpart considered in \cite{GaHS}.
This is due to the fact that  $\xi(t)$ has zero average, see \eqref{osc},
so that, unlike \cite{GaHS},  the motion of $\mathscr B$ produces no wake structure in the flow.
(ii) The non-autonomous character
of the principal linear part, where the drift term
$\eta(t)\cdot\nabla w$ cannot be seen as a perturbation to the main (Stokes) operator, for all sufficiently large times. In order to overcome the difficulty in (i) we adapt to the case at hand  the duality method developed by Yamazaki \cite{Y} that allows
 us to handle the additional linear terms
$h(t)(v\cdot\nabla w+w\cdot\nabla v)$ in \eqref{perturbed}, in spite of the ``poor" summability of $v$ at large distances. As far as the other difficulty, we shall employ the theory recently developed in \cite{Hi18, Hi} by the second author, which provides $L^q$-$L^r$
decay estimates of the evolution operator,
 $\{T(t,s)\}_{t\geq s\geq 0}$,
generated by the {\em non-autonomous} Oseen operator
$-P[\Delta+\eta(t)\cdot\nabla]$ --with $P$  Helmholtz projection on the space of
$L^q$-vector fields-- entirely analogous to those available in the  autonomous case for Stokes and Oseen semigroups \cite{I}, \cite{KS}, \cite{MSo}.

By suitably combining the above arguments and using the results in \cite{Ga}, in the present paper we are able to show, in particular, the  decay to 0 of $w(t)$, as $t\to\infty$, in appropriate $L^q$-spaces; see Theorem \ref{main}.
Moreover, by {\color{black} developing} 
an idea of Koba \cite{Ko}, we shall also show the decay of $w(t)$ in $L^\infty$-norm
(see \eqref{decay} below). However, our proof --based on the $L^\infty$-estimate of the composite 
operator
$T(t,s)P\mbox{div}$ given in Proposition \ref{composite}--
turns out to be  simpler and more direct than that given in \cite{Ko}.
\par
The plan of the paper is as follows. In the next section we shall state the main results, collected in Theorem \ref{main}. In Section 3
we present some results from \cite{Ga} and \cite{Hi18, Hi} and deduce some relevant consequences. 
The final Section 4 is devoted the proof of Theorem \ref{main}.
\medskip\par\noindent
{\bf Notation.} $C_{0,\sigma}^\infty(\Omega)$ is the subclass of vector functions $u$ in $C_0^\infty(\Omega)$ with $\mbox{div}\,u=0$. 
By $L^q ({\color{black} \Omega})$,  $1\leq q \leq \infty,$  
$W^{m,q}({\Omega}),$  $m \geq 0,$  $(W^{0,q}\equiv  L^q$), we denote usual Lebesgue and Sobolev classes of vector functions, with corresponding norms $\|.\|_{q}$ and $\|.\|_{m,q}$. Also,  $L^q_\sigma(\Omega)$ denotes the completion of $C_{0,\sigma}^\infty(\Omega)$ in  $L^q(\Omega)$, and $P: L^q\mapsto L^q_\sigma$ the associated Helmholtz projection {\color{black} (\cite{Ga-b}, \cite{M}, \cite{SS})}.
For $1 < p < \infty$
and $1 \le q \le\infty$, let $L^{p,q}(\Omega)$ denote the Lorentz space with norm $\|.\|_{p,q}$;
{\color{black} see \cite{BL} for details about this space. Since $P$ defines a bounded operator on $L^{p,q}(\Omega)$, we}
set $L^{p,q}_\sigma(\Omega)=P\,[L^{p,q}(\Omega)]$. Moreover, $D^{m,2}(\Omega)$ stands for the space of (equivalence classes of) functions 
${\color{black} u \in L^1_{loc}(\Omega)}$ such that
$ 
\sum_{|k|=m}\|D^k u\|_{2}<\infty\,.
$ Obviously, the latter defines a seminorm in $D^{m,2}$.
Let $B$ be a function space {\color{black} of spatial variable}
endowed with seminorm $\|\cdot\|_B$. For $r=[1,\infty]$, $\mathcal T>0$, $L^r(B)$ is the class of functions
$u:(0,\mathcal T)\rightarrow B$ such that 
$$
\|u\|_{L^r(B)}\equiv\left\{\begin{array}{ll}\smallskip\big( \Int{0}{\mathcal T}\|u(t)\|_B^r \big)^{\frac 1r}<\infty, \ \ \mbox{if 
$r\in [1,\infty)\,;$}\\   
\essup{t\in[0,\mathcal T]}\,\|u(t)\|_B <\infty, \ \ \mbox{if $r=\infty.$}
\end{array}\right.
$$
Likewise, we put
$$
W^{m,r}(B)=\Big\{u\in L^{r}(B): {\partial_t^ku\in L^{r}(B), \, k=1,\ldots,m}\Big\}\,.
$$\par
\section{Statement of Main Results}
By use of the evolution operator $T(t,s)$ mentioned in the introductory section,
problem \eqref{perturbed} is transformed into the integral equation
\begin{equation}
w(t)=w_0(t)-\int_0^t
T(t,s)P\mbox{div $(Fw)$}(s)\,ds
\label{int-eq}
\end{equation}
with
\begin{equation}
w_0(t)=\int_0^t T(t,s)Pf(s)\,ds,
\label{top}
\end{equation}
\begin{equation}
Fw=F_vw=w\otimes w+h(w\otimes v+v\otimes w).
\label{nonlinear}
\end{equation}
The main result reads
\begin{theorem}
Suppose \eqref{osc} and \eqref{transi} hold and let $|h^\prime|_0:=\sup_{t\geq 0}|h^\prime(t)|$.
For every $\varepsilon\in (0,\frac{1}{4})$, there is a constant
$\delta=\delta(\varepsilon)$ such that if
\begin{equation}
\|\xi\|_{W^{3,2}(0,{\cal T})}
\leq\frac{\delta}{1+|h^\prime|_0}
\label{small}
\end{equation}
then problem \eqref{int-eq} admits a unique solution
$w\in C_{w^*}((0,\infty); L^{3,\infty}_\sigma(\Omega))$
with the following properties:

\begin{enumerate}
\item
The equation \eqref{int-eq} is satisfied in $L^{3,\infty}_\sigma(\Omega)$.

\item
The initial condition:
\begin{equation}
\lim_{t\to 0}\|w(t)\|_{3,\infty}=0.
\label{IC}
\end{equation}

\item
There is a constant $C>0$ such that
\begin{equation}
\|w(t)\|_{3,\infty}\leq C\|\xi\|_{W^{3,2}(0,{\cal T})}
\label{bdd}
\end{equation}
for all $t\geq 0$.

\item
Attainability:
\begin{equation}
\begin{split}
&\|w(t)\|_q=
\left\{
\begin{array}{ll}
O\left(t^{-1/2+3/2q}\right), \qquad
&q\in (3,q_0), \\
O\left(t^{-1/2+\varepsilon}\right),
&q\in (q_0,\infty],
\end{array}
\right.  \\
&\|w(t)\|_{q_0,\infty}=O(t^{-1/2+\varepsilon}),
\end{split}
\label{decay}
\end{equation}
as $t\to\infty$, where $q_0=3/2\varepsilon$.

\end{enumerate}
\label{main}
\end{theorem}

\begin{remark}The unique existence of the evolution operator $T(t,s)$ or,
in other words, the well-posedness of the initial boundary value problem 
for the linearized system, 
was successfully proved by Hansel and Rhandi \cite{HR} even in the case when
the body $\mathscr B$ rotates.
The key point of their argument is how to overcome difficulties 
due to the rotational term;
in fact, the Tanabe--Sobolevskii theory   \cite{T} 
of parabolic evolution operators does not work in this situation.
\end{remark}\begin{remark}
We apply the theory in {\color{black} \cite{T}} to the non-autonomous Oseen operator 
without rotation. Thus, in such a case, 
the regularity properties of $T(t,s)$   basically coincide
with those of analytic semigroups
for the autonomous case.
As a consequence, one could show that the solution $w(t)$  in Theorem \ref{main} 
becomes  ``strong"
provided only $h^\prime(t)$, in addition to satisfying \eqref{transi}, is H\"older continuous.  
We will not give details of such a claim, 
since our main objective  is to show the attainability property.
\end{remark}
\begin{remark} We observe that our approach furnishes, in particular, also the {\em stability} of the time-periodic solution $v(t)$. In fact, this property can be established by studying an integral equation of the type \eqref{int-eq} obtained by  setting formally $h(t)\equiv 1$ (which implies that  
the term $f$ in \eqref{force} vanishes identically) and  replacing the function $w_0(t)$  with $\widetilde{w_0}(t)= T(t,0)w(0)$, where $w(0)$ is the initial perturbation. 
One can slightly modify the proof of Theorem \ref{main} to show that the  asymptotic decay property of $w(t)$ stated 
in {\color{black} \eqref{decay}} 
continues to hold, provided, in addition to \eqref{small}, that $w(0)\in L^{3,\infty}_\sigma(\Omega)$ with sufficiently small norm. 
\end{remark}

\section{Preparatory Results}
\label{preli}

Let us begin to recall the following result concerning the existence, uniqueness and asymptotic spatial behavior of solutions to \eqref{periodic-NS}. 
%
%
\begin{proposition}
[\cite{Ga}]
Let $\xi$ satisfy \eqref{osc}.
Then, there exists a constant $\varepsilon_0>0$ such that if
\begin{equation}
{\tt D}:=\|\xi\|_{W^{3,2}(0,{\cal T})}<\varepsilon_0,
\label{small-0}
\end{equation}
problem \eqref{periodic-NS} has one and only one time-periodic solution
$(v,p_v)$ of period ${\cal T}$ in the class
\begin{equation*}
\begin{split}
&v\in W^{2,2}(D^{2,2})\cap W^{1,2}(D^{4,2})\cap W^{2,\infty}(W^{1,2})
\cap L^\infty(D^{3,2}),  \\
&p_v\in L^\infty(W^{1,2})\cap W^{1,2}(D^{3,2}),
\end{split}
\end{equation*}
with all  corresponding norms of $(v,p_v)$  bounded from above 
by ${\tt D}$.
Moreover, there exists a constant $C>0$ such that this solution obeys the following estimates
\begin{equation}
\begin{split}
(1+|x|)|v(x,t)|
&+(1+|x|^2)\{|\nabla v(x,t)|+|p_v(x,t)|\}  \\
&+(1+|x|^3)\{|\nabla^2v(x,t)|+|\nabla p_v(x,t)|\}
\leq C\,{\tt D},
\label{pointwise}
\end{split}
\end{equation}
for all $(x,t)\in \Omega\times\mathbb R/{\cal T}\mathbb Z$.
\label{periodic}
\end{proposition}

\begin{remark}The constant $\delta$ in \eqref{small} of Theorem \ref{main}
must be taken smaller than $\varepsilon_0$ in \eqref{small-0}.\end{remark}

The next result regards the  large time behavior of the evolution operator
$T(t,s)$ and its  adjoint $T(t,s)^*$. These properties, among others, have been established in \cite{Hi18, Hi}. 
\begin{proposition}
[\cite{Hi18, Hi}]
Let $m\in (0,\infty)$ and assume
\begin{equation}
\sup_{t\geq 0}|\eta(t)|
+\sup_{t>s\geq 0}\frac{|\eta(t)-\eta(s)|}{t-s}
\leq m.
\label{rigid-bound}
\end{equation}

\begin{enumerate}
\item
Let $1<q<\infty$ and $q\leq r\leq\infty$.
Then, there is a constant $C=C(m,q,r,\Omega)>0$ such that
\begin{equation}
\|T(t,s)f\|_r\leq C(t-s)^{-(3/q-3/r)/2}\|f\|_q
\label{LqLr}
\end{equation}
for all $t>s\geq 0$, $f\in L^q_\sigma(\Omega)$ and that
\begin{equation}
\|T(t,s)f\|_{r,\infty}\leq C(t-s)^{-(3/q-3/r)/2}\|f\|_{q,\infty}
\label{evo-lorentz}
\end{equation}
for all $t>s\geq 0$ and $f\in L^{q,\infty}_\sigma(\Omega)$.

\item
Let $1<q\leq r\leq 3$.
Then there is a constant $C=C(m,q,r,\Omega)>0$ such that
\begin{equation}
\|\nabla T(t,s)^*g\|_r\leq C(t-s)^{-(3/q-3/r)/2-1/2}\|g\|_q
\label{adj-grad}
\end{equation}
for all $t>s\geq 0$, $g\in L^q_\sigma(\Omega)$ and that
\begin{equation}
\|\nabla T(t,s)^*g\|_{r,1}\leq C(t-s)^{-(3/q-3/r)/2-1/2}\|g\|_{q,1}
\label{adj-lorentz}
\end{equation}
for all $t>s\geq 0$ and $g\in L^{q,1}_\sigma(\Omega)$.
If in particular $1/q-1/r=1/3$ as well as 
$1<q {\color{black} <} r\leq 3$,
then there is a constant $C=C(m,q,\Omega)>0$ such that
\begin{equation}
\int_0^t\|\nabla T(t,s)^*g\|_{r,1}\,ds
\leq C\|g\|_{q,1}
\label{adj-sharp}
\end{equation}
for all $t>0$ and $g\in L^{q,1}_\sigma(\Omega)$

\end{enumerate}
\label{evo-op}
\end{proposition}

\begin{remark}
In \cite{Hi18, Hi} the assumption on $\eta$ is made in terms of the 
H\"older seminorm,
that is controlled by the left-hand side of \eqref{rigid-bound}, 
which is, in turn,  controlled by ${\tt D}$; see \eqref{small-0}.
Estimate \eqref{evo-lorentz} with $r<\infty$ immediately
follows from \eqref{LqLr} by interpolation.
The proof of $L^{q,\infty}$-$L^\infty$ estimate, that is, \eqref{evo-lorentz} 
with $r=\infty$, is not given in \cite{Hi18, Hi},
but it can be easily proved by use of the semigroup property,  following the lines of the proof
of \eqref{compo}--\eqref{compo2} below with $r=\infty$.
The remaining three bounds  \eqref{adj-grad}--\eqref{adj-sharp} are shown in \cite{Hi}. 
 However, we emphasize that \eqref{adj-lorentz} with  $r=3$
does not follow directly from \eqref{adj-grad} by interpolation.
The idea of deducing \eqref{adj-sharp} from \eqref{adj-lorentz}
is, in fact, due to Yamazaki \cite{Y}.\end{remark}

We next prove an important consequence of the previous proposition.
\begin{proposition}
Let $m\in (0,\infty)$ and assume \eqref{rigid-bound}. The following properties hold.

\begin{enumerate}
\item
Let $3/2\leq q<\infty$ and $q\leq r\leq\infty$.
Then there is a constant $C=C(m,q,r,\Omega)>0$ such that
the composite operator $T(t,s)P\mbox{\em div}$ extends to a bounded operator
from $L^q(\Omega)^{3\times 3}$ to $L^r_\sigma(\Omega)$, $r<\infty$,
and to $L^\infty(\Omega)^3$ subject to estimate
\begin{equation}
\|T(t,s)P\mbox{\em div $F$}\|_r
\leq C(t-s)^{-(3/q-3/r)/2-1/2}\|F\|_q
\label{compo}
\end{equation}
for all $t>s\geq 0$ and $F\in L^q(\Omega)^{3\times 3}$.

\item
Let $3/2<q<r\leq\infty$.
Then there is a constant $C=C(m,q,r,\Omega)>0$ such that
the composite operator $T(t,s)P\mbox{\em div}$ extends to a bounded operator
from $L^{q,\infty}(\Omega)^{3\times 3}$ to $L^r_\sigma(\Omega)$, $r<\infty$,
and to $L^\infty(\Omega)^3$ subject to estimate
\begin{equation}
\|T(t,s)P\mbox{\em div $F$}\|_r
\leq C(t-s)^{-(3/q-3/r)/2-1/2}\|F\|_{q,\infty}
\label{compo2}
\end{equation}
for all $t>s\geq 0$ and $F\in L^{q,\infty}(\Omega)^{3\times 3}$.

\end{enumerate}
\label{composite}
\end{proposition}

\begin{proof}
By density, it suffices to show \eqref{compo} for $F\in C_0^\infty(\Omega)^{3\times 3}$.
We first consider the case $3/2\leq q\leq r<\infty$, so that
$1<r^\prime\leq q^\prime\leq 3$.
By \eqref{adj-grad} we have
\begin{equation*}
\begin{split}
|\langle T(t,s)P\mbox{div $F$}, \varphi\rangle|
&=|\langle F, \nabla T(t,s)^*\varphi\rangle|  \\
&\leq \|F\|_q \|\nabla T(t,s)^*\varphi\|_{q^\prime}  \\
&\leq C(t-s)^{-(3/q-3/r)/2-1/2}\|F\|_q\|\varphi\|_{r^\prime}
\end{split}
\end{equation*}
for all $t>s\geq 0$ and $\varphi \in L^{r^\prime}_\sigma(\Omega)$,
which leads to \eqref{compo} with $r<\infty$.
This combined with \eqref{LqLr} ($r=\infty$) implies that
\begin{equation*}
\begin{split}
\|T(t,s)P\mbox{div}F\|_\infty
&\leq C(t-s)^{-3/4q}\|T((t+s)/2,s)P\mbox{div}F\|_{2q}  \\
&\leq C(t-s)^{-3/2q-1/2}\|F\|_q
\end{split}
\end{equation*}
yielding \eqref{compo} with $r=\infty$.

Let $3/2<q\leq r<\infty$, then \eqref{compo} implies
\begin{equation}
\|T(t,s)P\mbox{div $F$}\|_{r,\infty}
\leq C(t-s)^{-(3/q-3/r)/2-1/2}\|F\|_{q,\infty}
\label{compo3}
\end{equation}
for all $t>s\geq 0$ and $F\in L^{q,\infty}(\Omega)^{3\times 3}$.
Since
\begin{equation}
\|u\|_r\leq C\|u\|_{r_0,\infty}^{1-\theta}\|u\|_{r_1,\infty}^\theta
\label{interpo}
\end{equation}
where $1/r=(1-\theta)/r_0+\theta/r_1$ as well as $0<\theta<1$ and
$1<r_0<r<r_1\leq\infty$,
we obtain \eqref{compo2} from \eqref{compo3} as long as
$3/2<q<r<\infty$.
This combined with \eqref{LqLr} ($r=\infty$) leads to \eqref{compo2}
when $3/2<q<r=\infty$.
The proof is complete.
\end{proof}

\section{Proof of Theorem \ref{main}}
\label{pr}

Following Yamazaki \cite{Y}, we consider the following weak form of 
\eqref{int-eq}:
\begin{equation}
\langle w(t),\varphi\rangle
=\langle w_0(t),\varphi\rangle
+\int_0^t\langle (Fw)(s),
\nabla T(t,s)^*\varphi\rangle\,ds \qquad
\forall\varphi\in C_{0,\sigma}^\infty(\Omega).
\label{weak-int}
\end{equation}
For $q\in [3,\infty)$, let us introduce the space
\begin{equation*}
\begin{split}
X_q=\{w\in 
C_{w^*}((0,\infty);\,
L^{3,\infty}_\sigma(\Omega)\cap L^{q,\infty}_\sigma(\Omega));\,
&[w]_3+[w]_q<\infty, \\
&\lim_{t\to 0}\|w(t)\|_{3,\infty}=0
\},
\end{split}
\end{equation*}
where
\begin{equation}
[w]_q:=\sup_{t>0}t^{1/2-3/2q}\|w(t)\|_{q,\infty}.
\label{t-norm}
\end{equation}
Clearly, $X_q$ becomes a Banach space when endowed with norm
$[w]_3+[w]_q$.

Under the smallness condition \eqref{small-0}, the solution $v$ obtained 
in Proposition \ref{periodic}
and the force $f$ defined by \eqref{force} fulfill
\[
v(t),\,f(t)\in L^{3,\infty}(\Omega)\cap L^\infty(\Omega)
\]
with
\begin{equation}
\begin{split}
&\sup_{t\geq 0}\, (\|v(t)\|_{3,\infty}+\|v(t)\|_\infty)
\leq C{\tt D}, \\
&\sup_{t\geq 0}\, (\|f(t)\|_{3,\infty}+\|f(t)\|_\infty)
\leq C(|h^\prime|_0+{\tt D}){\tt D}, 
\end{split}
\label{est-periodic}
\end{equation}
which immediately follows from \eqref{pointwise}.
This, combined with \eqref{evo-lorentz}, implies the following lemma.
\begin{lemma}
Suppose \eqref{osc}, \eqref{transi} and \eqref{small-0}.
Then the function $w_0$ defined by \eqref{top}
belongs to $X_q$ for every $q\in [3,\infty)$.
Moreover, we have
$w_0(t)\in L^\infty(\Omega)$ for each $t>0$. Finally, 
for every $r\in [3,\infty]$, there is a constant $c_r>0$ such that
\begin{equation}
\|w_0(t)\|_{r,\infty}
\leq c_r(|h^\prime|_0+{\tt D}){\tt D}\,(1+t)^{-1/2+3/2r}
\label{est-top}
\end{equation}
for all $t>0$,
with ${\tt D}$ given in \eqref{small-0}.
\label{head}
\end{lemma}

\begin{proof}
Let $0\leq t<t+\tau$, then we have
\begin{equation*}
\begin{split}
&\quad w_0(t+\tau)-w_0(t) \\
&=\int_0^t\{T(t+\tau,s)-T(t,s)\}Pf(s)\,ds
+\int_t^{t+\tau}T(t+\tau,s)Pf(s)\,ds
=:I+II.
\end{split}
\end{equation*}
By \eqref{evo-lorentz} and \eqref{est-periodic}, we know that
\[
\|T(t,s)Pf(s)\|_{q,\infty}\leq C(|h^\prime|_0+{\tt D}){\tt D}=:C_0
\]
with some constant $C=C(q)>0$ independent of $(t,s)$
for every $q\in [3,\infty)$.
From the Lebesgue convergence theorem we infer
$I\to 0$ as $\tau\to 0$, whereas it follows at once 
$II\leq C_0\tau$.
For the other case $0<t/2<t+\tau<t$,
we have
\[
w_0(t+\tau)-w_0(t)
=\int_0^{t+\tau}\{T(t+\tau,s)-T(t,s)\}Pf(s)\,ds
-\int_{t+\tau}^t T(t,s)Pf(s)\,ds
\]
which goes to zero as $\tau\to 0$ by the same reasoning as above.
Consequently, $w_0(t)$ is even strongly continuous up to $t=0$ with values in 
$L^{q,\infty}(\Omega)$
as well as
$\|w_0(t)\|_{3,\infty}\to 0$ ($t\to 0$).
Concerning the estimate in $L^{r,\infty}(\Omega)$ with $r\in [3,\infty]$,
we consider only the one involving 
$\|w_0(t)\|_\infty$, 
since the other ones are obtained similarly.
Since $f(t)=0$ for $t\geq 1$, we use \eqref{evo-lorentz} to find
\[
\|w_0(t)\|_\infty\leq C\int_0^1(t-s)^{-1/2}\|Pf(s)\|_{3,\infty}\,ds
\leq Ct^{-1/2}(|h^\prime|_0+{\tt D}){\tt D}
\]
for $t\geq 2$, while we have
\[
\|w_0(t)\|_\infty\leq Ct^{1/2}(|h^\prime|_0+{\tt D}){\tt D}
\]
for $t<2$.
We thus obtain the desired estimate.
\end{proof}

Let us begin to prove the uniqueness property.
In fact, the solution obtained in Theorem \ref{main} is unique
in the sense of the following lemma, provided we choose the constant 
$\delta$ in \eqref{small}  smaller than the constant $\delta_0$ defined below.
\begin{lemma}
There is a constant $\delta_0>0$ such that if
${\tt D}\leq\delta_0$, then the solution to \eqref{weak-int} is unique in the ball 
$\{w\in X_3; [w]_3\leq\delta_0\}$.
\label{unique}
\end{lemma}

\begin{proof}
Let both $w,\,\widetilde w\in X_3$ satisfy \eqref{weak-int}.
By duality $L^{3,1}_\sigma(\Omega)^*=L^{3/2,\infty}_\sigma(\Omega)$ 
together with the weak-H\"older inequality,
we have
\[
|\langle w(t)-\widetilde w(t), \varphi\rangle|
\leq C([w]_3+[\widetilde w]_3+[v]_3)
[w-\widetilde w]_3
\int_0^t\|\nabla T(t,s)^*\varphi\|_{3,1}\,ds
\]
for all $\varphi\in C^\infty_{0,\sigma}(\Omega)$.
We employ \eqref{adj-sharp} and \eqref{est-periodic} to obtain
\[
[w-\widetilde w]_3
\leq c_*([w]_3+[\widetilde w]_3+{\tt D})[w-\widetilde w]_3
\]
by duality,
which yields the assertion by taking
$\delta_0=1/4c_*$.
\end{proof}

Given $\varepsilon\in (0,\frac{1}{4})$, we set
$q_0=3/2\varepsilon\in (6,\infty)$ and intend to find a solution $w\in X_{q_0}$
to \eqref{weak-int} provided ${\tt D}$ is small enough.
Given $w\in X_{q_0}$ and $t>0$, we define $(\Psi w)(t)$ by
\[
\langle(\Psi w)(t),\varphi\rangle
=\int_0^t\langle (Fw)(s), \nabla T(t,s)^*\varphi\rangle\,ds \qquad
\forall\varphi\in C_{0,\sigma}^\infty(\Omega).
\]
We then find
\begin{equation}
\begin{split}
&[\Psi w]_3\leq C([w]_3+[v]_3)[w]_3, \\
&[\Psi w]_{q_0}\leq C([w]_3+[v]_3)[w]_{q_0}.
\end{split}
\label{into}
\end{equation}
The former is deduced along the same lines as in Lemma \ref{unique}, 
while the latter is verified by splitting the integral as
\begin{equation}
\left(\int_0^{t/2}+\int_{t/2}^t\right)
s^{-1/2+3/2q_0}\|\nabla T(t,s)^*\varphi\|_{r,1}\,ds
=:I+II
\label{split}
\end{equation}
where $r\in (3/2,2)$ is determined by the condition $1/r=2/3-1/q_0$.
In fact, in view of \eqref{adj-lorentz}, we get
\begin{equation*}
I\leq C\int_0^{t/2}s^{-1/2+3/2q_0}(t-s)^{-1}\,ds \,\|\varphi\|_{q_0^\prime,1}
\end{equation*}
that leads to the desired estimate, where $1/q_0^\prime+1/q_0=1$. Also,
 employing \eqref{adj-sharp}, we show
\[
II\leq Ct^{-1/2+3/2q_0}\int_{t/2}^t\|\nabla T(t,s)^*\varphi\|_{r,1}\,ds.
\]
By the same token we can show 
\begin{equation}
\begin{split}
[\Psi w-\Psi\widetilde w]_r
\leq C([w]_3+[\widetilde w]_3+[v]_3)
[w-\widetilde w]_r \qquad r\in \{3,q_0\},
\end{split}
\label{contra}
\end{equation}
for all $w,\, \widetilde w\in X_{q_0}$.

The above computations  are exactly the same as in \cite[Section 8]{HiS}.
However, because in our case the equation is non-autonomous,
the argument to show the continuity with respect to time is
different from the one adopted by Yamazaki \cite[Section 3]{Y} in which
the strong continuity is deduced for $t>0$.
Here, we show merely the weak* continuity for $t>0$, while we still have
strong convergence to 0 at the initial time, namely,
\[
\|(\Psi w)(t)\|_{3,\infty}
\leq C([w]_3+[v]_3)\sup_{0<s<t}\|w(s)\|_{3,\infty}\to 0
\]
as $t\to 0$ (as well as the same property for $w_0(t)$; see Lemma \ref{head}).
Actually, for $0<t<t+\tau$ and $\varphi\in C^\infty_{0,\sigma}(\Omega)$,
let us consider
\begin{equation}
\begin{split}
\langle (\Psi w)(t+\tau)-(\Psi w)(t),\varphi\rangle
&=\int_0^t
\langle (Fw)(s), \nabla\{T(t+\tau,s)^*-T(t,s)^*\}\varphi\rangle\,ds \\
&\quad +\int_t^{t+\tau}
\langle (Fw)(s), \nabla T(t+\tau,s)^*\varphi\rangle\,ds  \\
&=:I+II.
\end{split}
\label{weak-conti}
\end{equation}
Let $r\in (3/2,2)$ be the same exponent as in \eqref{split}.
By using the backward semigroup property we have
\begin{equation*}
\begin{split}
I&\leq C([w]_3+[v]_3)[w]_{q_0}\int_0^t
s^{-1/2+3/2q_0}\|\nabla T(t,s)^*\{T(t+\tau,t)^*\varphi-\varphi\}\|_{r,1}\,ds \\
&\leq C([w]_3+[v]_3)[w]_{q_0}\|T(t+\tau,t)^*\varphi-\varphi\|_{3/2,1}
\end{split}
\end{equation*}
which goes to zero as $\tau\to 0$ for all 
$\varphi\in C^\infty_{0,\sigma}(\Omega)$.
Concerning the other part, we have
\begin{equation*}
\begin{split}
II&\leq C([w]_3+[v]_3)[w]_{q_0}\int_t^{t+\tau}
s^{-1/2+3/2q_0}\|\nabla T(t+\tau,s)^*\varphi\|_{r,1}\,ds \\
&\leq 
C([w]_3+[v]_3)[w]_{q_0}t^{-1/2+3/2q_0}\tau^{1/2-3/2q_0}\|\varphi\|_{3/2,1}
\end{split}
\end{equation*}
for all $\varphi\in C^\infty_{0,\sigma}(\Omega)$,
which implies the strong convergence  with values in 
$L^{3,\infty}_\sigma(\Omega)$ also of this part.
Summing up, by density argument, we can state that the left-hand side of \eqref{weak-conti} goes to zero as $\tau\to 0$ for all
$\varphi\in L^{3/2,1}_\sigma(\Omega)$ and 
{\color{black} also for all $\varphi\in L^{q_0^\prime,1}_\sigma(\Omega)$ in view of \eqref{into}.}
The case $0<t/2<t+\tau<t$ is similarly discussed with
\begin{equation*}
\begin{split}
&\quad \langle (\Psi w)(t+\tau)-(\Psi w)(t), \varphi\rangle  \\
&=\int_0^{t+\tau}
\langle (Fw)(s), \nabla\{T(t+\tau,s)^*-T(t,s)^*\}\varphi\rangle\,ds  \\
&\quad -\int_{t+\tau}^t \langle (Fw)(s), \nabla T(t,s)^*\varphi\rangle\,ds
\end{split}
\end{equation*}
to conclude that 
$\Psi w$ is weak* continuous with values in $L^{3,\infty}_\sigma(\Omega)$
and in $L^{q_0,\infty}_\sigma(\Omega)$.

By these results, we can then  conclude that $w_0+\Psi w\in X_{q_0}$, for every $w\in X_{q_0}$.
Assume now ${\tt D}\leq 1$.
By taking into account 
\eqref{est-periodic}, \eqref{est-top}, \eqref{into} and \eqref{contra}, one can easily show the existence of a fixed point $w\in X_{q_0}$ of the map
\[
w\mapsto w_0+\Psi w
\]
in a closed ball of $X_{q_0}$ with radius
$2(c_3+c_{q_0})(|h^\prime|_0+1){\tt D}$, 
provided  $(|h^\prime|_0+1){\tt D}$ is small enough, where the smallness
depends on $\varepsilon$ (recall that $q_0=3/2\varepsilon>6$).
By Lemma \ref{unique} it is the only solution to \eqref{weak-int}
in the small within $X_3$.
From the interpolation inequality \eqref{interpo}, 
the solution $w(t)$ satisfies \eqref{decay} for $q\in (3,q_0)$.

For the solution $w(t)$ constructed above, 
it follows from \eqref{compo2} and \eqref{est-periodic}
that the second term on the right-hand side of \eqref{int-eq}
is Bochner integrable with values in $L^3(\Omega)$; in fact,
\[
\int_0^t\|T(t,s)P\mbox{div $(Fw)$}(s)\|_3\,ds
\leq C([w]_3+[v]_3)[w]_{q_0}
\]
for all $t>0$.
The latter, in conjunction with  Lemma \ref{head}, shows that
the weak form \eqref{weak-int} leads, in fact,  to the the conclusion that  the integral equation \eqref{int-eq}  is meaningful  
in $L^{3,\infty}_\sigma(\Omega)$.
Actually, from the computations that we shall perform in the next paragraph, it turns out 
that the second term of the right-hand side of \eqref{int-eq} is also 
Bochner integrable in $L^\infty(\Omega)$.

It remains to show \eqref{decay} for the other case 
$q\in (q_0,\infty]$, $q_0=3/2\varepsilon$.
To this end, on the account of the interpolation inequality  \eqref{interpo}, it is enough to prove the decay of $w(t)$ in the $L^\infty$-norm.
The argument that follows is essentially due to Koba \cite{Ko},
but, unlike \cite{Ko}, we shall not use a duality procedure; rather, we will directly apply the $L^{q,\infty}$-$L^\infty$ 
estimate of the composite operator
$T(t,s)P\mbox{div}$ proved in Proposition \ref{composite}. As a consequence,  the proof is considerably shortened and more direct.
By means  of \eqref{compo2}, it is easily seen that
\[
\int_0^t\|T(t,s)P\mbox{div $(w\otimes w)$}(s)\|_\infty\,ds 
\leq C[w]_{q_0}^2 t^{-1/2}
\]
for all $t>0$, where the summability of the integral is ensured since $q_0>6$.
Let $t>2$.
We split the other part of the integral of \eqref{int-eq} into two parts
\[
\left(\int_0^{t-1}+\int_{t-1}^t\right)
\|T(t,s)P\mbox{div $[h(w\otimes v+v\otimes w)]$}(s)\|_\infty\,ds
=:I+II.
\]
We utilize \eqref{compo2} again to find that
\[
I\leq \int_0^{t-1}
(t-s)^{-1-\varepsilon}\|v(s)\|_{3,\infty}\|w(s)\|_{q_0,\infty}\,ds
=\int_0^{t/2}+\int_{t/2}^{t-1}
=:I_1+I_2
\]
with
\[
I_1\leq C{\tt D}[w]_{q_0}t^{-1/2}, \qquad
I_2\leq C{\tt D}[w]_{q_0}t^{-1/2+\varepsilon},
\]
and that
\[
II\leq \int_{t-1}^t
(t-s)^{-(3/r+3/q_0)/2-1/2}\|v({\color{black} s})\|_{r,\infty}\|w(s)\|_{q_0,\infty}\,ds
\leq C{\tt D}[w]_{q_0}t^{-1/2+\varepsilon}
\]
where $r\in (3,\infty)$ is chosen in such a way that
$1/r+1/q_0<1/3$, see \eqref{est-periodic}.
The proof is complete.

\end{document}